\documentclass[twoside]{amsart}%
\usepackage{mathtools}
\usepackage{latexsym}
\usepackage{mathrsfs}
\usepackage{pstricks}
\usepackage{listings}

\usepackage{enumitem}
\usepackage{todonotes}
\usepackage{amssymb,bm}
\usepackage{amsmath}
\usepackage{amsfonts}
\usepackage{amsthm}
\usepackage{tikz}
\usepackage{epsfig}
\usepackage{verbatim}
\usepackage{float}
\usepackage{tabularx}

\providecommand{\U}[1]{\protect\rule{.1in}{.1in}}

\newcolumntype{Y}{>{\raggedleft\arraybackslash}X}
\def\bc{{\mathbb{C}}}

\def\bn{{\mathbb{N}}}

\def\br{{\mathbb{R}}}

\def\bz{{\mathbb{Z}}}

\def\br{\mathbb R}

\def\wt{\widetilde}

\def\vs{\vskip.3cm}
\def\noi{\noindent}
\def\wt{\widetilde}
\def\gdeg{G\text{\rm -deg}}
\def\caldeg{\mathcal G\text{\rm -deg}}

\def\Om{\Omega}

\usepackage{listings}
  \lstdefinelanguage{GAP}{
    basicstyle=\ttfamily,
    keywords={true, false, function, return, fail, if, in, while, do, od, else, elif, fi, break, continue},
    keywordstyle=\color{blue}\bfseries,
    otherkeywords={% Operators
      >, <, ==
    },
    identifierstyle=\color{black},
    sensitive=True,
    comment=[l]{\#},
    commentstyle=\color{cyan},
    stringstyle=\color{red},
    morestring=[b]',
    morestring=[b]"
  }

\def\ve{\varepsilon}

\DeclareMathOperator{\id}{Id}

\newcommand\cW{\ensuremath{\mathcal W}}

\newcommand\bbR{\ensuremath{\mathbb R}}

\DeclareMathAlphabet{\mathscrbf}{OMS}{mdugm}{b}{n}

\newtheorem{theorem}{Theorem}[section]
\newtheorem{proposition}[theorem]{Proposition}
\newtheorem{lemma}[theorem]{Lemma}

{\bfseries}{\rmfamily}
\newtheorem{definition}[theorem]{Definition}
\newtheorem{remark}[theorem]{Remark}

\newtheorem{remark-definition}[theorem]{Remark and Definition}

\begin{document}

\title[Unbounded Branches of Non-Radial Solutions]{Unbounded Branches of Non-Radial Solutions to Semilinear Elliptic Systems on a Disc and their Patterns}

\author[Z. Ghanem --- C. Crane]{Ziad Ghanem --- Casey Crane}

\address
{\textsc{Ziad Ghanem}\\
Department of Mathematical Sciences\\
University of Texas at Dallas\\
Richardson, TX 75080, USA}

\email{ziad.ghanem@utallas.edu}

\address
{\textsc{Casey Crane}\\
Department of Mathematical Sciences\\
University of Texas at Dallas\\
Richardson, TX 75080, USA.}
\email{casey.crane@utallas.edu}



\begin{abstract}
In this paper, we leverage 
the $O(2) \times \mathbb Z$-equivariant Leray-Schauder degree and a novel characterization of the Burnside Ring $A(O(2) \times \bz_2)$ presented by Ghanem in \cite{Ghanem1} to obtain $(\rm i)$ an existence result for non-radial solutions to the problem $-\Delta u = f(z,u) + Au$, $u|_{\partial D} = 0$ and $(\rm ii)$ local and global bifurcation results for multiple branches of non-radial solutions to the one-parameter family of equations $-\Delta u = f(z,u) + \textbf{A}(\alpha)u$, $u|_{\partial D} = 0$, where $D$ is the planar unit disc, $u(z) \in \mathbb R^N$, $A : \br^N \rightarrow \mathbb R^N$ is an $N \times N$ matrix, $\textbf{A}: \br \rightarrow L(\br^N)$ is a continuous family of $N \times N$ matrices and $f: \overline D \times \mathbb R^N \rightarrow \mathbb R^N$ is a sublinear, $O(2) \times \mathbb Z_2$-equivariant function of order $o(|u|)$ as $u$ approaches the origin in $\mathbb R^N$.
\end{abstract}

\maketitle

%---

\section{Introduction} 
Although a proven tool for studying the existence of solutions to a wide variety of differential equations (cf. \cite{Amster, Bebernes, Gaines, Hartman, Knobloch, Mahwin3, Mahwin5}), the inability of the Leray-Schauder degree to distinguish $(\rm i)$ between solutions that differ by a fixed angular shift and $(\rm ii)$  radial solutions from non-radial solutions has hindered its suitability for differential equations defined on radially symmetric domains. For example, let $D:=\{z\in \bc: |z|<1\}$ be the planar unit disc and consider the following system of elliptic equations
\begin{equation}\label{eq:Lap}
\begin{cases}
-\triangle u= f(z,u) + Au, \quad u(z)\in V, \\
u|_{\partial D}=0, 
\end{cases}
\end{equation}
where $V:= \br^N$, $A:V \rightarrow V$ is an $N \times N$ matrix and $f:\overline D\times V\to V$ is a continuous function satisfying the conditions:
  \begin{enumerate}[label=($A_\arabic*$)]
\item\label{c1} $f(z,u)$ is $o(|u|)$ as $u$ approaches $0$ for all $z \in \overline{D}$, i.e.
\[
    \lim_{u \to 0} \frac{f(z,u)}{|u|} = 0\quad z\in \overline{D};
\]
\item\label{c2} there exist $a$, $b>0$ and $\nu\in (0,1)$ such that
\begin{align*}
%\label{eq:s1-2}
		 |f(z,u)|<a|u|^\nu +b \quad \text{for all} \quad
		 {z\in \overline{D}}, {u\in V}.
\end{align*}
\item\label{c3} $f$ is odd, i.e.
\[
f(z,-u) = -f(z,u) \quad \text{for all} \quad
		 {z\in \overline{D}}, {u\in V}.
\]
\item\label{c4} 
$f(e^{i\theta}z,u)=f(z,u)$ for all $z\in D$, $u\in V$ and $\theta\in \br$.	
\end{enumerate}
Condition \ref{c1} 
ensures that the linearization of $f$ at the origin in $V$ does not interact with the spectrum of $A$, condition \ref{c2} is necessary to guarantee \textit{a priori} bounds on the solutions to \eqref{eq:Lap}
and the conditions \ref{c3}--\ref{c4} imply that the system \eqref{eq:Lap} admits the symmetries of the product group
\[
G := O(2) \times \bz_2.
\]
Whereas the only solutions to \eqref{eq:Lap} ascertainable by the classical Leray-Schauder degree have no dependence on the angular variable $\theta \in [0,2\pi]$, the effectiveness of the $G$-equivariant Leray-Schauder degree in detecting the existence of solutions not fixed by the action of $O(2)$ to \eqref{eq:Lap}  problems similar to \eqref{eq:Lap} has already been well established (cf. \cite{BalHooton, BalChen, book-new, Balkraw, AED, Dab, DuanCrane, Duan, Eze, Carlos, Carloskraw, Ghanem1, Krawcewicz}).
Indeed, Balanov et. al obtain in \cite{BalHooton} a sufficient condition for the existence of non-radial solutions to a boundary value problem of the form \eqref{eq:Lap} involving the 
parity of the number of eigenvalues of the matrix $A$, accounting for geometric multiplicity, which dominate the positive zeros of each of the Bessel functions of the first kind. Relying on a novel characterization of the interactions between orbit types in the Burnside Ring $A(G)$ (see Appendix \ref{sec:appendix} for a definition of the Burnside Ring) described by Ghanem in \cite{Ghanem1}, we are able to contribute an additional statement specifying the minimal radial symmetries of these solutions up to a class of maximal subgroup conjugacy classes in the isotropy lattice $\Phi_1(G)$ (see Remark \ref{rm:amalgamated_notation} for the definition of the amalgamated subgroups $D_{2m} {}^{D_m}\times^{\bz_1} \bz_2$ with $m \in \bn$). The reader may compare the following result with Theorem $3.1$ in \cite{BalHooton}:
\begin{theorem}
Let $m > 0$ be a positive Fourier mode and assume that the spectrum of $A$ does not contain the squares of the positive zeros of any of the Bessel functions of the first kind. If the matrix $A: V \to V$ has an {\bf odd number} of eigenvalues with {\bf odd geometric multiplicity} that are larger than an {\bf odd number} of the positive zeros of the $m$-th Bessel function of the first kind, then equation \eqref{eq:Lap} admits a non-trivial solution $u \in \mathscr{H} \setminus \{0\}$ with isotropy subgroup satisfying $(G_u) \geq (D_{2m} {}^{D_m}\times^{\bz_1} \bz_2)$.
\end{theorem}
Besides its efficacy in proving the existence of non-radial solutions to radially symmetric differential equations, the $G$-equivariant Leray-Schauder degree has also been used to study bifurcation problems with domain-inherent symmetries (cf. \cite{BalBurnett, book-new, AED, BalWu, Ghanem1, Ghanem}). For example, in  the article \cite{Ghanem}, Ghanem et al. use a $G$-equivariant Leray-Schauder degree based framework to determine the existence of branches of non-trivial solutions bifurcating from the trivial solution and establish sufficient conditions for such branches $(\rm i)$ to consist only of non-radial solutions and $(\rm ii)$ to be unbounded for the following parameterization of
system \eqref{eq:Lap}:
\begin{equation}\label{eq:Lap_parameterized}
\begin{cases}
-\triangle u= f(z,u) + \bm A(\alpha)u, \quad u(z)\in V, \\
u|_{\partial D}=0, 
\end{cases}
\end{equation}
where $\alpha \in \br$ is a bifurcation parameter, $\bm A: \br \rightarrow L(V)$ is a continuous family of $N \times N$ matrices and where $f$ and $V$ are as they were in \eqref{eq:Lap}. However, the results obtained in \cite{Ghanem} are only capable of detecting the emergence of at most a single unbounded branch of non-radial solutions at each critical point $(\alpha_0,0)$ (see Definition \ref{def:critical_point}). Depending again on the novel characterization of $A(G)$ found in \cite{Ghanem1}, we are able to here present a result describing the emergence of branches of non-trivial solutions across all possible Fourier modes. Compare the following local equivariant bifurcation result with Theorem $4.1$, Corollary $4.2$ and Proposition $5.1$ in \cite{Ghanem}:
\begin{theorem}
Let $m > 0$ be a positive Fourier mode and assume that $(\alpha_0, 0) \in \Lambda$ is an isolated critical point of \eqref{eq:Lap_parameterized} with a deleted regular neighborhood $\alpha_0^- < \alpha_0 < \alpha_0^+$. If the 
number of eigenvalues with odd geometric multiplicity that are larger than an odd number of the positive zeros of the $m$-th Bessel function of the first kind {\bf differ in parity} for the matrices $\bm A(\alpha_0^\pm): V \to V$, then \eqref{eq:Lap_parameterized} admits a branch of non-trivial solutions $\mathscr{C}$ with branching point $(\alpha_0, 0)$ satisfying $(G_u) \geq (D_{2m} {}^{D_m}\times^{\bz_1} \bz_2)$ for all $u \in \mathscr{C}$.
\end{theorem}
Moreover, the global bifurcation result presented in \cite{Ghanem} $(\rm i)$ requires a monotonic dependence of the eigenvalues of $A: \br \rightarrow L(V)$ on the bifurcation parameter $\alpha$, $(\rm ii)$ assumes the existence of only finitely many critical points and $(\rm iii)$ involves degree invariant computations at every critical point of \eqref{eq:Lap_parameterized}. We are able to determine the global properties and ensure non-radial symmetries in branches of solutions bifurcating from every odd Fourier mode without any monotonicity assumptions and involving only a single degree invariant computation. The following global equivariant bifurcation result might be compared with Theorem $4.3$ and Proposition $5.2$ in \cite{Ghanem}:
\begin{theorem}
Let $m > 0$ be an {\bf odd}, positive Fourier mode and assume that \eqref{eq:Lap_parameterized} admits a finite number of critical points $\{(\alpha_k, 0)\}_{k=0}^N$ enumerated such that, if $i< j$, then $\alpha_i < \alpha_j$ for all $i,j \in \{0,1,\ldots,N\}$ and with associated deleted regular neighborhoods. If the 
number of eigenvalues with odd geometric multiplicity that are larger than an odd number of the positive zeros of the $m$-th Bessel function of the first kind differ in parity for the matrices $\bm A(\alpha_0^-): V \to V$ and $A(\alpha_N^+): V \to V$, then \eqref{eq:Lap_parameterized} admits an {\bf unbounded} branch of {\bf non-radial} solutions $\mathscr{C}$ with branching point $(\alpha_0, 0)$ satisfying $(G_u) \geq (D_{2m} {}^{D_m}\times^{\bz_1} \bz_2)$ for all $u \in \mathscr{C}$.
\end{theorem}
The remainder of the paper is organized as follows:
In Section \ref{sec:functional_space}, we prepare the problems 
\eqref{eq:Lap} and  \eqref{eq:Lap_parameterized} for application of the $G$-equivariant Leray-Schauder degree. This involves functional reformulation in an appropriate Sobolev space, allowing us to derive practical formulae for a pair of degree invariants, each defined in terms of the Burnside ring products of a finite number of basic degrees associated with the irreducible $G$-representations, whose non-triviality is equivalent to the existence of non-trivial solutions to \eqref{eq:Lap} and the emergence of branches of non-trivial solutions to \eqref{eq:Lap_parameterized}, respectively. Since every finite product of Burnside ring elements can be expressed in terms of a finite number of products between Burnside ring generators, we demonstrate in Sections \ref{sec:existence_results} and \ref{sec:bifurcation_results}, how the novel characterization of $A(G)$ presented in \cite{Ghanem1} leads to our main existence and bifurcation results.
\section{Functional Space Reformulation} \label{sec:functional_space}
A prerequisite for the application of any $G$-equivariant degree-theory based argument to the problems \eqref{eq:Lap} and \eqref{eq:Lap_parameterized}
is their reformulation as fixed point equations in an appropriate functional $G$-space with a nonlinear operator in the form of a $G$-equivariant compact perturbation of identity.
Following \cite{BalHooton}, we consider the Sobolev space $\mathscr H:=H_{0}^2(D,\br^k)$ equipped with the usual norm
\[
\|u\|:=\max\{\|D^s u\|_{L^2}: |s|\le 2\}, \quad s:=(s_1,s_2), \; D^s u:=\frac{\partial ^{|s|}}{\partial^{s_1}x\partial^{s_2}y},
\]
and with the isometric $G$-action
\begin{align}\label{def:G_action}
(e^{i\vartheta},\pm 1)u(r,\theta) := \pm u(r, \theta + \vartheta), \quad (\kappa,\pm 1)u(r,\theta) := \pm u(r, -\theta), \quad   u \in \mathscr H, 
\end{align}
where $\vartheta \in [0,2 \pi]$ and $\kappa$ is the reflection generator in $O(2)$. 
We also consider the Laplacian operator 
\[
\mathscr L:\mathscr H\to L^2(D;\br^k), \quad \mathscr Lu:=-\triangle u,
\]
the Nemytski operator
\begin{align*}
%\label{eq:operator-N}
N:L^q(D;\bbR^N) \to L^2(D;\bbR^N), \quad N(v)(z) := f(z,v(z)), \quad z\in \overline D,\quad \alpha \in \bbR,
\end{align*}
and the Sobolev embedding
\[
j:\mathscr H\to L^q(D;\br^k), \quad j(u(z)) := u(z),
\]
these last two defined for any $q > \max\{1,2\nu\}$ (for example, it is enough to take $q:= 2\beta$, cf. assumptions (A3) and (A4)). 
\vs
Notice that \eqref{eq:Lap} can now be reformulated as the fixed point equation
\[
\mathscr L u = N(j(u)) + Aj(u),
\]
and \eqref{eq:Lap_parameterized} as the equation
\[
\mathscr L u = N(j(u)) + \bm A(\alpha)j(u).
\]
%\begin{remark} \label{rm:operators}
%    Notice that $\mathscr L$ is a linear isomorphism, $N$ is bounded and $j$ is compact.
%\end{remark}
\section{On the Existence of Non-Radial Solutions to \eqref{eq:Lap}} \label{sec:existence_results}
Since $\mathscr L$ is a linear isomorphism, $N$ is bounded and $j$ is compact, the operator 
\begin{align} \label{def:operator_F}
\mathscr F: \mathscr H \to \mathscr H, \quad \mathscr F(u) := u - \mathscr L^{-1}(N(j(u)) + A j(u)), 
\end{align}
is a compact perturbation of the identity. Notice that system \eqref{eq:Lap} is equivalent to the operator equation
\begin{equation}\label{eq:operator-eq}
\mathscr F(u) = 0,
\end{equation}
in the sense that in the sense that $u \in \mathscr H$ is a solution to \eqref{eq:Lap} if and only if it satisfies \eqref{eq:operator-eq}.
In what follows, we will denote by
\begin{align*}
%\label{def:operator_A}
    \mathscr A: \mathscr H \rightarrow \mathscr H, \quad D\mathscr F(0)=:\mathscr A = \id - \mathscr L^{-1}(A \circ j),
\end{align*}
the linearization of \eqref{def:operator_F} at the origin in $\mathscr H$. Under the conditions \ref{c1}---\ref{c3} and provided that $\mathscr A$ is an isomorphism, Balanov et. al 
establish $(\rm i)$ that $\mathscr F$ is a $G$-equivariant completely continuous field with respect to the $G$-action \eqref{def:G_action}, $(\rm ii)$ that all of the non-trivial solutions to \eqref{eq:Lap} can be confined to an annulus of the form $\Om := B_R(\mathscr H) \setminus B_{\varepsilon}(\mathscr H) \subset \mathscr H$ for some sufficiently small choice of $\varepsilon > 0$ and sufficiently large choice of $R > 0$ and $(\rm iii)$ that $(\mathscr F,\Om)$ constitutes an admissible $G$-pair, such that the existence of non-trivial solutions to \eqref{eq:Lap} is equivalent to non-triviality of the degree invariant
\[
\gdeg(\mathscr F,\Om) = (G) - \gdeg(\mathscr A, B_1(\mathscr H)),
\]
where $\gdeg$ is the $G$-equivariant Leray-Schauder degree and $(G) \in A(G)$ is the unit element of the Burnside Ring (see Appendix \ref{sec:appendix} for definition of the Burnside Ring $A(G)$, definition of $G$-admissibility/the set of all admissible $G$-pairs $\mathscr M^G$ and for definition of the $G$-equivariant  degree Leray-Schauder degree $\gdeg:\mathscr M^G \rightarrow A(G)$ and cf. \cite{BalHooton} for proof of the above statement). Consequently, we have the following sufficient condition for the existence of a non-trivial solution to \eqref{eq:Lap}:
\begin{lemma}\label{lemm:sufficient_condition}
If for some orbit type $(H) \in \Phi_0(G) \setminus \{(G) \}$ one has
\[
\operatorname{coeff}^{H}\left( \gdeg(\mathscr A, B(\mathscr H)) \right) \neq 0,
\]
then there exists a nontrivial solution to \eqref{eq:Lap} $u \in \mathscr H$ with an isotropy subgroup $G_u \leq G$ satisfying $(G_u) \geq (H)$.  
\end{lemma}
In this way, the problem of finding non-trivial solutions to \eqref{eq:Lap} has been reformulated as the problem of computing the coefficients of non-unit orbit types (i.e. orbit types belonging to the set $\Phi_0(G) \setminus \{(G)\}$) in the Burnside Ring element $\gdeg(\mathscr A, B(\mathscr H))$. Since the $G$-equivariant Leray-Schauder degree provides a full equivariant classification to the solution set of the equation \eqref{eq:operator-eq}, Lemma \ref{lemm:sufficient_condition}, moreover, represents a framework for identifying the minimal spatio-temporal symmetries of any predicted solution to \eqref{eq:Lap}.
\subsection{Towards a Computational Formula for the Degree $\gdeg(\mathscr A, B(\mathscr H))$}
In order to effectively make use of Lemma \ref{lemm:sufficient_condition} for proving the existence of non-trivial solutions to \eqref{eq:Lap}, we must construct a computational formula for the degree calculation $\gdeg(\mathscr A, B(\mathscr H)) \in A(G)$. Our first step in this direction is to describe the $G$-isotypic decomposition of $\mathscr H$ (for an approachable exposition of the isotypic decomposition and the methods employed in this section in general, we refer our readers to the survey paper \cite{survey}). As will be seen in Section \ref{sec:bifurcation_results}, this decomposition of our functional space in terms of a direct sum of irreducible $G$-representations will also prove useful our analysis of the equivariant bifurcation problem \eqref{eq:Lap_parameterized}.
\vs
Let's begin by describing the set of all irreducible $G$-representations. For each $m \in \bn$, we denote by $\mathcal W_m \simeq \bc$ the irreducible $O(2)$-representation equipped with the $m$-folded $O(2)$-action
\begin{align*}
    e^{i \vartheta}w:= e^{i m \vartheta}w, \quad \kappa w:= \pm \bar{w},  \quad \theta \in SO(2), \; \kappa \in O(2), \; w \in \mathcal W_m,
\end{align*}
by $\mathcal W_0 \simeq \br$ the irreducible $O(2)$-representation on which $O(2)$ acts trivially and by $\{ \mathcal W_m^- \}_{m \geq 0}$ the corresponding list of irreducible $O(2) \times \bz_2$-representations, where the superscript is meant to indicate that each of the irreducible $O(2)$-representations has been equipped with the antipodal $\mathbb{Z}_2$-action. 
\vs
We notice, alongside Balanov et al in \cite{BalHooton}, that the spectrum of the Laplacian operator is given in terms of the positive zeros $\{ \sqrt{s_{m,n}}\}_{n \in \bn}$ of the Bessel functions of the first kind $\{J_m:\br \rightarrow \br\}_{m \in \bn \cup \{0\}}$ as follows
\[
\sigma(\mathscr L)= \{s_{m,n} : (m,n) \in \bn \cup \{0\} \times \bn \},
\]
and that each of the associated eigenspaces
\[
\mathscr E_{m,n} := \{J_m(\sqrt{s_{m,n}}r) (\cos(m\theta)a + \sin(m\theta)b): a,b \in V\},
\]
may be equipped with corresponding $G$-action
\[
(e^{i\vartheta},\pm1)u(r,\theta) := \pm u(r,\theta + m \vartheta), \quad u \in \mathscr E_{m,n},
\]
such that
\[
\mathscr E_{m,n} \simeq \mathcal W_m^- \otimes V, \quad m \in \bn_0.
\]
Consequently, $\mathscr H$ admits the $G$-isotypic decomposition
\begin{align} \label{eq:G_isotypic_decomp}
  \mathscr H := \overline{\bigoplus_{m = 0}^{\infty} \mathscr H_{m}}, \quad \mathscr H_{m} := \overline{\bigoplus_{n=1}^{\infty} \mathscr E_{m,n}},  
\end{align}
where the closure is taken in $\mathscr H$. To be clear, every $G$-isotypic component $\mathscr H_m$ is modeled on the irreducible $G$-representation $\mathcal W_{m}^-$.
\vs
As a $G$-equivariant linear operator, $\mathscr A:\mathscr H \rightarrow \mathscr H$ respects the decomposition \eqref{eq:G_isotypic_decomp} in the sense that
$\mathscr A(\mathscr H_m) \subset \mathscr H_m$ for all $m \geq 0$. 
Adopting the notation,
\[
\mathscr A_{m} := \mathscr A|_{\mathscr H_{m}}: \mathscr H_{m} \rightarrow \mathscr H_{m},
\]
it follows that the spectrum of $\mathscr A$ is given by
\begin{align*} %\label{mathscrA_spectral_decomposition1}
 \sigma (\mathscr A)= \bigcup\limits_{m=0}^{\infty}  \sigma(\mathscr A_{m}), \quad \sigma(\mathscr A_{m}) = \{ 1 - \frac{\mu_j}{s_{m,n}} : \mu_j \in \sigma(A), \; n \in \bn\}.
\end{align*}
Enumerating the spectrum of $A$
\[
\sigma(A) = \{\mu_1, \mu_2, \ldots, \mu_r\},
\]
and indicating by $m_j$ the geometric multiplicity of the corresponding eigenvalue $\mu_j \in \sigma(A)$, each eigenvalue of $\mathscr A$ can be identified with an index triple $(m,n,j) \in \bn \cup \{0\} \times \bn \times \{1,\ldots,r\}$ as follows
\begin{align}\label{def:operator_eigenvalues}
\mu_{m,n,j} := 1 - \frac{\mu_j}{s_{m,n}} = \frac{s_{m,n} - \mu_j}{s_{m,n}}.
\end{align}
Therefore, we can ensure the invertibility of $\mathscr A$
with the non-degeneracy assumption:
\begin{enumerate}[label=($D$)]
    \item\label{d} For all $\mu \in \sigma(A)$, $m \in \bn \cup \{0\}$ and $n \in \bn$, one has $\mu \neq s_{m,n}$.
\end{enumerate}
Under the assumption \ref{d}, the product property of the $G$-equivariant Leray-Schauder degree (cf. Appendix \ref{sec:appendix}) permits us to express the degree $\gdeg(\mathscr A, B(\mathscr H))$ in terms of a Burnside Ring product of the $G$-equivariant Leray-Schauder degrees of the Fourier mode restrictions $\mathscr A_{m}: \mathscr H_{m} \rightarrow \mathscr H_{m}$ of the $G$-equivariant linear isomorphism $\mathscr A: \mathscr H \rightarrow \mathscr H$ to the open unit balls $B(\mathscr  H_{m}):= \{ u \in \mathscr H_{m} \; : \; \| u \|_{\mathscr H} < 1 \}$ as follows
\begin{equation}\label{eq:gdegA_product_property_decomp}
  \gdeg(\mathscr A, B(\mathscr H)) = \prod\limits_{m=0}^\infty \gdeg( \mathscr A_{m}, B(\mathscr H_{m})).  
\end{equation}
On the other hand, each degree $\gdeg( \mathscr A_{m}, B(\mathscr H_{m}))$ is fully specified by the corresponding negative spectra $\sigma_-(\mathscr A_{m}) := \{ \mu_{m,n,j} < 0 : n \in \bn, \; j \in \{0,1,\ldots,r\}\}$ according to the formula
\begin{equation}\label{eq:negative_spectrum_gdegA}
\gdeg(\mathscr A_{m}, B( \mathscr H_{m})) = \prod_{\mu_{m,n,j} \in \sigma_-(\mathscr A_{m})} (\deg_{\mathcal W^-_{m}})^{m_j}
\end{equation}
where $\deg_{\mathcal W^-_m} \in A(G)$ is the basic degree associated with the irreducible $G$-representation $\mathcal W^-_{m}$ (cf. Appendix \ref{sec:appendix}) and $(G) \in A(G)$ is the unit element of the Burnside 
Ring. 
Putting together \eqref{eq:gdegA_product_property_decomp} and \eqref{eq:negative_spectrum_gdegA}, we introduce some notations to keep track of the indices 
\begin{align*}
%\label{index_1}
   \Sigma := \left\{ (m,n,j) : m \in \mathbb{N} \cup \{0\}, \; n \in \bn, \; j \in \{0,1,\ldots,r \} \right\},
\end{align*}
which contribute non-trivially to $\gdeg(\mathscr A, B(\mathscr H))$. Specifically, the \textit{negative} spectrum of $\mathscr A: \mathscr H \rightarrow \mathscr H$ is accounted for with the index set
\begin{align}
\label{def:index_set_Sigma}
   \Sigma_0 := \left\{ (m,n,j) \in \Sigma :  \mu_{m,n,j} < 0 \right\}.
\end{align}
Combining this with formulas \eqref{eq:gdegA_product_property_decomp} and \eqref{eq:negative_spectrum_gdegA} yields
\begin{align} \label{eq:computational_formula_gdegA}
    \gdeg(\mathscr A, B(\mathscr H)) \; = \prod\limits_{(m,n,j) \in \Sigma_0} (\deg_{\mathcal W^-_{m}})^{m_j}.
\end{align}
\subsection{Basic Degrees and Multiplication in the Burnside Ring $A(G)$}
All that remains is to describe a strategy for identifying orbit types in the Burnside Ring product of a finite number of basic degrees.
\vs
In \cite{BalHooton}, Balanov et al. notice that the basic degree associated with the irreducible $G$-representation $\mathcal W_{m}^-$ is given by 
\begin{align*}
\deg_{\mathcal W_{m}^-} =
\begin{cases}
    (G) - (O(2) \times \bz_1) & \text{ if } m = 0; \\
    (G) - (D_{2m} {}^{D_m}\times^{\bz_1} \bz_2) & \text{ if } m  > 0, 
\end{cases}
\end{align*}
where $O(2) \times \bz_1 \leq G$ is the standard product group and the notation
$D_{2m} {}^{D_m}\times^{\bz_1} \bz_2 \leq G$ is used to indicate the set of all pairs $(x,y) \in D_{2m} \times \bz_2$ for which the orders of $x \in D_{2m}$ and $y \in \bz_2$ share the same parity, i.e. the amalgamated subgroup
\[
D_{2m} {}^{D_m}\times^{\bz_1} \bz_2 := \{ (x,y) \in D_{2m} \times \bz_2 : 2 \mid (|x| + |y|) \}.
\]
\begin{remark}\label{rm:amalgamated_notation}
Amalgamated notation is a short-hand for easy identification of the subgroups of a product group such as $G$. For more details on this topic, the reader may consult the manuscripts \cite{book-new, AED} or the appendices of either of the articles \cite{BalHooton, Ghanem}.
\end{remark}

For the sake of simplifying our exposition throughout the remainder of this paper, we adopt the notation 
\begin{align} \label{def:maximal_orbit_type}
(H_{m}) := (D_{2m} {}^{D_m}\times^{\bz_1} \bz_2), \quad m > 0. 
\end{align}
We will also employ the {\it Iverson brackets} which map any logical predicate $P$ to the set $\{0,1\}$ according to the rule
\begin{align} \label{notation_iverson}
    [P] := \begin{cases}
        1 \quad & \text{ if } P \text{ is true}; \\
        0 \quad & \text{ otherwise},
    \end{cases}
\end{align}
and, for any finite set of natural numbers $I \subset \bn$, indicate by $\mathcal B(I)$ the Boolean expression
\begin{align}
\mathcal B(I) \equiv \text{ `` } \text{Every pair of indices } x,y \in I \text{ satisfies either one of the relations } 2 | \frac{(x\pm y)}{\gcd(x,y)}. \text{"}
\end{align}
Since the orbit type $(H_{m})$ is a maximal element in the isotropy lattice $\Phi_0(G; \mathscr H_{m}\setminus\{0\})$, its relations to other orbit types with respect to the natural ordering of $\Phi_0(G)$ can be fully characterized as follows
\[
(H_{m}) \leq (K) \iff (K) = (G) \text{ or } (K) = (H_{m'}) \text{ with } m \leq m' \text{ and } \mathcal B(\{m,m'\}).
\]
It follows, from the recurrence formula (cf. Appendix \eqref{sec:appendix}) for multiplication in the Burnside Ring, that one has
\begin{align*}
\operatorname{coeff}^{H_{s}}((H_{m}) \cdot (K)) = 
\begin{cases}
1 & \text{ if } (K) = (G) \text{ and } s = m; \\
2 & \text{ if } (K) = (H_{m'}), \; s = \gcd(m,m') \text{ and } \mathcal B(\{m,m'\}); \\
%2 & \text{ if } j = 1, \; (K) = (H_{k,1}), \; 2 \mid (m+k)\text{ and } s = \gcd(m,k); \\
0 & \text{ otherwise},
\end{cases}
\end{align*}
equivalently
\begin{align*}
\operatorname{coeff}^{H_{s}}((H_{m}) \cdot (K)) = [(K) = (G)][s=m] + 2[(K) = (H_{m'})][s = \gcd(m,m')][\mathcal B(\{m,m'\})].  
\end{align*}
With these preliminaries out of the way, we can now derive some results characterizing the behavior of the orbit type $(H_{m})$ in the Burnside Ring product of any finite collection of basic degrees
\begin{align} \label{def:finite_collection_bdeg}
    \left\{ \deg_{\mathcal W_{m_k}^-} \in A(G)  :  m_k \in \mathbb{N} \right\}_{k=1,2,\ldots,N}.
\end{align}
Let's begin with the observation that basic degrees are involutive elements in $A(G)$:
\begin{lemma} \label{lemm:involutive_bdeg}
For any $m \in \bn$, one has
    \[
    \deg_{\mathcal W_{m}^-} \cdot \deg_{\mathcal W_{m}^-} = (G).
    \]
\end{lemma}
\begin{proof}
The result follows from a direct computation of the relevant Burnside Ring product
    \begin{align*}
    \deg_{\mathcal W_{m}^-} \cdot \deg_{\mathcal W_{m}^-} &= ((G)-(H_{m})) \cdot ((G)-(H_{m})) \\
    &= (G) - (H_{m}) - (H_{m}) + 2(H_{m}).
    \end{align*}
\end{proof}
\vs
%This result naturally generalizes to the product of any basic degree $\deg_{\mathcal W_{m}^-}$ with itself a finite number of times
%\begin{corollary}
%    For any $m \in \bn$, one has
 %   \begin{align*}
 %    (\deg_{\mathcal W_{m}^-})^N = \begin{cases}
%         (G) & \text{ if } N \text{ is even}; \\
%         \deg_{\mathcal W_{m}^-} & \text{ if } N \text{ is odd}.
%     \end{cases} 
%    \end{align*}
%\end{corollary}
Next, let's characterize the coefficients of the orbit types $(H_{s})$ of the form \eqref{def:maximal_orbit_type} in the Burnside Ring product of any pair of
basic degrees $\deg_{\mathcal W_{m}^-},\deg_{\mathcal W_{m'}^-} \in A(G)$ with $m,m' \in \bn$:
\begin{lemma}
    For any $m,m' \in \bn$, one has
    \begin{align*}
    \operatorname{coeff}^{H_{s}}(\deg_{\mathcal W_{m}^-} \cdot \deg_{\mathcal W_{m'}^-}) = \begin{cases}
        2 & \text{ if } s \notin \{m,m'\}, \;  s = \gcd(m,m') \text{ and } \mathcal B(\{m,m'\}); \\
        1 & \text{ if } s \in \{m,m'\}, \;   s = \gcd(m,m') \text{ and } \mathcal B(\{m,m'\});  \\
        -1 & \text{ if } s \in \{m,m'\}  \text{ and } s \neq \gcd(m,m'); \\
        0 & \text{ if } s \notin \{m,m'\}  \text{ and } s \neq \gcd(m,m'),
    \end{cases}    
    \end{align*}
    or equivalently
    \begin{align*}
  \operatorname{coeff}^{H_{s}}(\deg_{\mathcal W_{m}^-} \cdot \deg_{\mathcal W_{m'}^-}) = -[s \in \{m,m'\}]  + 2[s = \gcd(m,m')][\mathcal B(\{m,m'\})].        
    \end{align*}
\end{lemma}
\begin{proof}
Again, the result follows from a direct computation of the relevant Burnside Ring product
    \begin{align*}
    \deg_{\mathcal W_{m}^-} \cdot \deg_{\mathcal W_{m'}^-} &= ((G)-(H_{m})) \cdot ((G)-(H_{m'})) \\
    &= (G) - (H_{m}) - (H_{m'}) + 2(H_{\gcd(m,m')})[\mathcal B(\{m,m'\})].
\end{align*}
\end{proof}
\vs
What follows is the full generalization of the previous results to an orbit type $(H_{m})$ of the form \eqref{def:maximal_orbit_type} in the Burnside Ring product of any finite collection of basic degrees $\{\deg_{\mathcal W_{m_k}^-} \}_{k=1}^N$ where $m_1,\ldots,m_{N} \in \mathbb{N}$  are distinct:
\begin{theorem}\label{thm:product_Nbdegs}
For any set of distinct indices $M := \{ m_1,m_2,\ldots m_N\}$, one has
\begin{align*} 
%\label{l3_res}
\operatorname{coeff}^{H_{m_0}} \left( \prod\limits_{k=1}^N \deg_{\mathcal W_{m_k}^-} \right)  =  -\left[m_0 \in M\right] + 2 \sum\limits_{\substack{ I \in \mathcal P(M) \\ I \neq \emptyset, \{m_0\} }} (-2)^{\vert I \vert-2} 
 \left[m_0 = \operatorname{gcd}(I)  \right]
 \left[ \mathcal B(I) \right]
    \end{align*}
\end{theorem}
\begin{proof}
The result follows directly from the observation that the Burnside Ring product of any finite set of basic degrees $\{\deg_{\mathcal W_{m_k}^-} \}_{k=1}^N$ has an expansion of the form
\begin{align*} %\label{prod_1}
    \prod\limits_{k=1}^N \deg_{\mathcal W_{m_k}^-} &= \prod\limits_{k=1}^N  (G) - (H_{m_k}) = \sum\limits_{I \in \mathcal P(M)} \prod\limits_{m \in I} -(H_{m}) = \sum\limits_{I \in \mathcal P(M)} (2)^{\vert I \vert -1} (-1)^{\vert I \vert} \mathcal B(I)(H_{\gcd(I)}),
\end{align*}
where $\mathcal P(M)$ is the 
power set and the expression $\sum_{I \in \mathcal P(M)}$ describes a summation over all subsets of the index set $M$, including the empty set, in which case we put $\prod_{m \in \emptyset} -(H_{m}) := (G)$, and the full set.
\end{proof}
\subsection{An Existence Result for Non-Radial Solutions to \eqref{eq:Lap}}\label{sec:nontriviality_deg}
We are almost ready to present our main existence result. First, we will need some notation to keep track of how often each basic degree $\deg_{\mathcal W_m^-}$ appears in the Burnside Ring product \eqref{eq:computational_formula_gdegA}. Specifically, for each Fourier mode $m \geq 0$, we put 
\begin{align} 
\begin{cases} \label{def:set_Sigma_m}
  \Sigma^m := \{ (n,j) : (m,n,j) \in \Sigma_0, \; 2 \nmid m_j \}; \\   
\mathfrak n^m := | \Sigma^m |,  
\end{cases}
\end{align}
and
\begin{align} \label{def:set_S}
 S:= \{ m \in \bn : 2 \nmid \mathfrak n^m \}.  
\end{align}
\begin{proposition} \label{prop:main_existence}
For any $m_0 \in \bn$, one has 
\begin{align*} 
\operatorname{coeff}^{H_{m_0}} & \left( \gdeg\left(\mathscr A,B(\mathscr H) \right) \right)  =  - \left[m_0 \in  S \right] +2 \sum\limits_{\substack{ I \in \mathcal P(S) \\ I \neq \emptyset, \{m_0\} }} (-2)^{\vert I \vert-2} 
\left[ \mathcal B_H(I) \right]\left[m_0 = \operatorname{gcd}(I)  \right]. 
    \end{align*}
\end{proposition}
\begin{proof}
Using the fact that every basic degree $\deg_{\mathcal W^-_m}$ is an involutive element in the basic degree (cf. Lemma \eqref{lemm:involutive_bdeg}), the computational formula can be improved to
\begin{align*}
\gdeg\left(\mathscr A,B(\mathscr H)\right) &= \prod_{m \geq 0} \prod_{(n,j) \in \Sigma^m} \deg_{\mathcal W_m^-} = \prod_{m \geq 0} (\deg_{\mathcal W_m^-})^{ \mathfrak n^m}  \\
&= (\deg_{\mathcal W_0^-})^{\mathfrak n^0} \cdot \prod_{m \in S} \deg_{\mathcal W_m^-}.
\end{align*}
On the other hand, since one has
\[
\operatorname{coeff}^{H_s}(\deg_{\mathcal W_0^-}) = 0, \quad s \in \bn,
\]
the coefficient standing next to $(H_{m_0})$ for any $m_0 \in \bn$ in the degree $\gdeg\left(\mathscr A,B(\mathscr H)\right)$ coincides with the coefficient standing next to $(H_{m_0})$ in the product of basic degrees $\prod_{m \in S} \deg_{\mathcal W^-_m}$, i.e.
\[
\operatorname{coeff}^{H_{m_0}}  \left( \gdeg\left(\mathscr A,B(\mathscr H) \right) \right) = \operatorname{coeff}^{H_{m_0}}  \left(\prod_{m \in S} \deg_{\mathcal W_m^-} \right).
\]
At this point, the result follows directly from Theorem \ref{thm:product_Nbdegs}.
\end{proof}
\section{Local and Global Bifurcation in 
\eqref{eq:Lap}} \label{sec:bifurcation_results}
We consider the family of compact perturbations of the identity operator
\begin{align} \label{def:operator_F_parameterized}
\mathscrbf F: \bbR \times \mathscr H \to \mathscr H, \quad \mathscrbf F(\alpha,u) := u - \mathscr L^{-1}(N(j(u)) + \bm A(\alpha)j(u)),
\end{align}
noticing that \eqref{eq:Lap_parameterized} is equivalent to the operator equation 
\begin{equation}\label{eq:operator-eq_param}
\mathscrbf F(\alpha,u) = 0,
\end{equation}
in the sense that $(\alpha, u) \in \br \times \mathscr H$ is a solution to \eqref{eq:Lap_parameterized} if and only if it satisfies 
\eqref{eq:operator-eq_param}. 
In what follows, we will denote by
\begin{align}\label{def:operator_A_parameterized}
    \mathscrbf A: \br \times \mathscr H \rightarrow \mathscr H, \quad D\mathscrbf F(0,\alpha)=:\mathscrbf A(\alpha) = \id - \mathscr L^{-1}(\bm A(\alpha) \circ j),
\end{align}
the linearization of \eqref{def:operator_F_parameterized} at the origin in $\mathscr H$. If we enumerate the spectrum of $\bm A(\alpha)$ as follows
\[
\sigma(\bm A(\alpha)) := \{ \bm \mu_1(\alpha), \bm \mu_2(\alpha), \ldots, \bm \mu_r(\alpha) \}
\]
then the spectrum of each $\mathscrbf A(\alpha): \mathscr H \rightarrow \mathscr H$ is given by
\[
\sigma(\mathscrbf A(\alpha)) = \left\{\bm \mu_{m,n,j}
(\alpha):= \frac{s_{m,n}-\bm \mu_j(\alpha)}{s_{m,n}} : \bm \mu_j(\alpha) \in \sigma(\bm A(\alpha)), \; m \in \bn \cup \{0\}, \; n \in \bn \right\}.
\]
Notice also that the set of all solutions to the operator equation \eqref{eq:operator-eq_param} can be divided into the set of trivial solutions 
\[
M:= \{(\alpha,0) \in \br \times \mathscr H \},
\]
and the set of non-trivial solutions
\[
\mathscr S:= \{(\alpha,u) \in \br \times \mathscr H : \mathscrbf F(\alpha,u) = 0, \; u \not\equiv 0\}.
\]
Given any orbit type $(H) \in \Phi_0(G)$, we can always consider the $H$-fixed-point set 
\[
\mathscr S^H := \{(\alpha,u) \in \mathscr S: G_u \leq H \},
\]
consisting of all non-trivial solutions to \eqref{eq:operator-eq_param} with {\it symmetries at least $(H)$}, i.e. $(\alpha,u) \in \mathscr S^H$ if and only if $\mathscrbf F(\alpha,u) = 0$, $u \in \mathscr H \setminus \{0\}$ and 
\[
h u(r,\theta) = u(r,\theta), \; \text{ for all } h \in H \text{ and } (r,\theta) \in D.
\]
\subsection{The Local Bifurcation Invariant and Krasnosel'skii's Theorem}\label{sec:local-bif-inv}
Formulation of a Krasnosel'skii type local bifurcation result for equation \eqref{eq:operator-eq_param} necessitates the introduction of additional notations and terminology
(for more details, the reader is referred to \cite{AED, book-new}). Our first definition clarifies what is meant by a bifurcation of the equation \eqref{eq:operator-eq_param}:
\begin{definition}\rm
 A trivial solution $(\alpha_0,0) \in M$ is said to be a \textit{bifurcation point} for the equation \eqref{eq:operator-eq_param} if every open neighborhood of the point $(\alpha_0,0)$ has non-trivial intersection with $\mathscr S$.   
\end{definition}
It is well-known that a necessary condition for any trivial solution $(\alpha_0,0) \in M$ to be a bifurcation point for the equation \eqref{eq:operator-eq_param} is that the linear operator $\mathscrbf A(\alpha_0):\mathscr H \rightarrow \mathscr H$ is not an isomorphism. This leads to the following definition:
\begin{definition} \label{def:critical_point}\rm
A trivial solution $(\alpha_0,0) \in M$ is said to be a \textit{regular point} for the equation \eqref{eq:operator-eq_param} if $\mathscrbf A(\alpha_0)$ is an isomorphism and a \textit{critical point} otherwise. Moreover, a critical point $(\alpha_0,0) \in M$ is said to be \textit{isolated} if there exists a deleted $\epsilon$-neighborhood $0< \vert \alpha - \alpha_0 \vert < \epsilon$ such that for all $\alpha \in (\alpha_0 - \epsilon, \alpha_0 + \epsilon) \setminus \{\alpha_0\}$, the point $(\alpha,0) \in M$ is regular.
\end{definition}
The set of all critical points for equation \eqref{eq:operator-eq_param}, denoted $\Lambda$, is called the 
{\it critical set}, i.e.
\begin{align}\label{eq:critical}
 \Lambda:=\{(\alpha,0): \text{ $\mathscrbf A(\alpha):\mathscr H \rightarrow \mathscr H$ is not an isomorphism}\}.   
\end{align}
\begin{definition}\rm
 A trivial solution $(\alpha_0,0) \in M$ is said to be a {\it branching point} for the equation \eqref{eq:operator-eq_param} if there exists a non-trivial continuum $K \subset \overline{\mathscr S}$ with $K \cap M = \{ (\alpha_0,0) \}$. In which case, we call any maximal connected set $\mathscr C \subset \overline{\mathscr S}$ containing $(\alpha_0,0)$ a \textit{branch} of nontrivial solutions bifurcating from the branching point $(\alpha_0,0)$.   
\end{definition}
Whereas the classical Krasnosiel'skii bifurcation result is only concerned with the existence of a branch of nontrivial solutions for the equation \eqref{eq:operator-eq_param} bifurcating from a given critical  point $(\alpha_0,0) \in \Lambda$, the equivariant Krasnosiel'skii bifurcation analogue, which we employ in this paper, is also concerned with the symmetric properties of the solutions belonging to such a branch.
\begin{definition} \rm
Given a subgroup $H \leq G$, denote by $\mathscr S^H$ the corresponding $H$-fixed point space of non-trivial solutions.
A branch of solutions $\mathscr C$ is said to have \textit{symmetries at least} $(H)$ if $\mathscr C \cap \overline{\mathscr S^H} \not=\emptyset$.  
\end{definition}
Let $(\alpha_0,0) \in \Lambda$ be an isolated critical point with a deleted $\epsilon$-neighborhood
\[
\{ \alpha \in \mathbb{R} : 0 < | \alpha - \alpha_0 | < \epsilon \},
\]
on which $\mathscrbf A(\alpha): \mathscr H \rightarrow \mathscr H$ is an isomorphism and choose $\alpha^\pm_0 \in (\alpha_0 - \ve, \alpha_0 + \ve)$ with $\alpha^-_0 \leq \alpha_0 \leq \alpha^+_0$. Since the two operators $\mathscrbf A(\alpha^\pm): \mathscr H \rightarrow \mathscr H$ are non-singular, there exists a number $\delta >0$ sufficiently small such that, adopting the notations $\mathscrbf F_{\pm}(u) := \mathscrbf F(\alpha^\pm_0, u)$, $B_{\delta} := \{u\in \mathscr H: \|u\|< \delta\}$, one has
\begin{enumerate}
    \item[$(\rm i)$] $\mathscrbf F_{\pm}^{-1}(0) \cap \partial B_{\delta} = \emptyset$,
    \item[$(\rm ii)$] $\mathscrbf F_{\pm}$ are $B_{\delta}$-admissibly $G$-homotopic to $\mathscrbf A(\alpha^\pm)$, respectively. 
\end{enumerate}
It follows, from the homotopy property of the $G$-equivariant Leray-Schauder degree (cf. Appendix \ref{sec:appendix}), that $(\mathscrbf F_{\pm},B_{\delta})$ are admissible $G$-pairs in $\mathscr H$ and also that $\gdeg(\mathscrbf F_{\pm},B_{\delta}) = \gdeg(\mathscrbf A(\alpha^\pm_0), B(\mathscr H))$, where $B(\mathscr H)$ is the open unit ball in $\mathscr H$. We 
call the Burnside Ring element
\begin{align} \label{def:local_bifurcation_invariant}
 \omega_{G}(\alpha_0):=\gdeg(\mathscrbf A(\alpha^-_0), B(\mathscr H))-\gdeg(\mathscrbf A(\alpha^+_0), B(\mathscr H)), 
\end{align}
the {\it local bifurcation invariant} at $(\lambda_0,0)$. The reader is referred to \cite{book-new, AED} for proof that the invariant \eqref{def:local_bifurcation_invariant} does not depend on the choice of $\alpha^\pm_0 \in \mathbb{R}$ or radius $\delta >0$, and also for the proof of the following local bifurcation result, which is a consequence of the equivariant version of a classical result of K. Kuratowski (cf. \cite{Kura}, Thm. 3, p. 170).
 \vs
 \begin{theorem}\em 
 {\bf(M.A. Krasnosel'skii-Type Local Bifurcation)}\label{th:Kras}
 Let $\mathscrbf F: \mathbb{R} \times \mathscr H \rightarrow \mathscr H$ be a completely continuous $G$-equivariant field 
 with a linearization $D \mathscrbf F(0) : \br \times \mathscr H \rightarrow \mathscr H$ admitting 
 an isolated critical point $(\alpha_0,0)$. If $\omega_{G}(\alpha_0) \neq 0$, then
 \begin{itemize}
     \item[(i)] there exists a branch of nontrivial solutions $\mathscr C$ to system \eqref{eq:operator-eq_param} with branching point $(\alpha_0,0)$;
     \item[(ii)] moreover, if $(H) \in \Phi_0(G)$ is an orbit type with
     \begin{align*}
     \operatorname{coeff}^{H}
         (\omega_{G}(\alpha_0)) \neq 0,
     \end{align*}
 \end{itemize}
then there exists a branch of non-trivial solutions bifurcating from $(\alpha_0,0)$ with symmetries at least $(H)$.
\end{theorem}
Adapting the notions \eqref{def:index_set_Sigma}, \eqref{def:set_S}--\eqref{def:set_Sigma_m} to our parameterized setting with the notations
\begin{align} \label{def:index_set_Sigma_parameterized}
 \Sigma_0(\alpha) := \{ (m,n,j) \in \Sigma : \bm \mu_{m,n,j}(\alpha)<0 \},
\end{align}
and
\begin{align} 
\begin{cases} \label{def:set_Sigma_m_parameterized}  
  \Sigma^m(\alpha) := \{ (n,j) : (m,n,j) \in \Sigma_0(\alpha), \; 2 \nmid m_j \}; \\   
\mathfrak n^m(\alpha) := | \Sigma^m(\alpha) |; \\
 S(\alpha):= \{ m \in \bn : 2 \nmid \mathfrak n^m(\alpha) \},
\end{cases}
\end{align}
and accounting for the change in cardinality of each of the index sets $\Sigma^m(\alpha)$ as $\alpha \in \br$ varies along the interval $[\alpha_0^-,\alpha_0^+]$ with the notations
\begin{align}\label{def:N^m}
\begin{cases}
    \mathfrak t^m(\alpha_0) := \mathfrak n^m(\alpha_0^+) - \mathfrak n^m(\alpha_0^-); \\
    J(\alpha_0) := \{ m \in \bn : 2 \nmid \mathfrak t^m(\alpha_0) \},
\end{cases}
\end{align}
we can exactly determine the coefficient standing next to any of the orbit types $(H_{m_0}) \in \Phi_0(G)$ in the local bifurcation invariant $\omega_G(\alpha_0) \in A(G)$ as follows:
\begin{proposition} \label{prop:main_local_bifurcation}
Let $(\alpha_0,0) \in \Lambda$ be an isolated critical point with a deleted regular neighborhood $\alpha^-_0 < \alpha_0 < \alpha^+_0$ on which $\mathscrbf A(\alpha)$ is an isomorphism. For any $m_0 > 0$, one has
\begin{align*} 
\operatorname{coeff}^{H_{m_0}} & \left( \omega_G(\alpha_0)  \right)  =  - \left[m_0 \in  J(\alpha_0) \right] (-1)^{[2 \nmid \mathfrak n^{m_0}(\alpha_0^+)]} \\
& \quad +2 \sum\limits_{\substack{ I \in \mathcal P(J(\alpha_0)) \\ I \neq \emptyset, \{m_0\} }} (-2)^{\vert I \vert-2} 
\left[ \mathcal B_H(I) \right]\left[m_0 = \operatorname{gcd}(I)  \right] (-1)^{[I \in \mathcal P(S(\alpha_0^+))]},
\end{align*}
\end{proposition}
\begin{proof}
Since $(\alpha^\pm_0,0) \in \mathbb{R} \times \mathscr H$ are regular points of \eqref{eq:operator-eq_param}, we can apply the computational formula \eqref{eq:computational_formula_gdegA} to each of the Leray-Schauder degrees $\gdeg(\mathscrbf A(\alpha_0^\pm),B(\mathscr H))$ to obtain a corresponding computational formula for the local bifurcation invariant \eqref{def:local_bifurcation_invariant}:
\begin{align}\label{eq:computational_formula_local_bifurcation_invariant}
     \omega_G(\alpha_0) = \prod\limits_{(n,m,j) \in \Sigma_0(\alpha^-_0)} (\deg_{\mathcal W_{m}^-})^{m_j} - \prod\limits_{(n,m,j) \in \Sigma_0(\alpha^+_0)} (\deg_{\mathcal W_{m}^-})^{m_j}.
\end{align}
Applying Proposition \ref{prop:main_existence} to each product of basic degrees separately, the coefficient standing next to $(H_{m_0})$ for any $m_0 \in \bn$ in \eqref{eq:computational_formula_local_bifurcation_invariant} is given by
\begin{align*} 
\operatorname{coeff}^{H_{m_0}}   \left( \omega_G(\alpha_0) \right) & =  - \left[m_0 \in  S(\alpha_0^-) \right] + 2\sum\limits_{\substack{ I \in \mathcal P(S(\alpha_0^-)) \\ I \neq \emptyset, \{m_0\} }} (-2)^{\vert I \vert-2} 
\left[ \mathcal B_H(I) \right]\left[m_0 = \operatorname{gcd}(I)  \right] \\
& \quad + \left[m_0 \in  S(\alpha_0^+) \right]  -2\sum\limits_{\substack{ I \in \mathcal P(S(\alpha_0^+)) \\ I \neq \emptyset, \{m_0\} }} (-2)^{\vert I \vert-2} 
\left[ \mathcal B_H(I) \right]\left[m_0 = \operatorname{gcd}(I)  \right].
\end{align*}
Notice that any index $m$ belonging to
$S(\alpha_0^-) \cap S(\alpha_0^+)$ contributes trivially to the coefficient of $(H_{m_0})$. On the other hand, one has $m \notin S(\alpha_0^-) \cap S(\alpha_0^+)$ if and only if $\mathfrak n^m(\alpha_0^\pm)$ differ in parity and 
the two cases $2 \nmid \mathfrak n^m(\alpha_0^-)$ or $2 \nmid \mathfrak  n^m(\alpha_0^+)$ are accounted for appropriate sign arguments.
\end{proof}
\subsection{The Rabinowitz Alternative}
In order to employ the Leray-Schauder $G$-equivariant degree to describe the global properties of branches of non-trivial solutions bifurcating from the critical points of the equation \eqref{eq:operator-eq_param}, we need to make an additional assumption: 
\begin{enumerate}[label=($B$)]
\item\label{b} The critical set $\Lambda \subset M$ (given by \eqref{eq:critical}) is discrete.
\end{enumerate}
Notice $(\rm i)$ that the local bifurcation invariant $\omega_{G}(\alpha_0)$ at any critical point $(\alpha_0,0) \in \Lambda$ is well-defined under assumption \ref{b} and $(\rm ii)$ that if $\mathcal U \subset \mathbb{R} \times \mathscr H$ is an open bounded $G$-invariant set, then its intersection with the critical set is finite. These observations are important to the statement of the following global bifurcation result, the proof of which can be found in \cite{book-new, AED}:
\vs
\begin{theorem}\label{th:Rabinowitz-alt}{\bf (The Rabinowitz Alternative)} \rm
Let $\mathcal U \subset \mathbb{R} \times \mathscr H$  be an open bounded $G$-invariant set with $\partial \mathcal U \cap \Lambda = \emptyset$. If $\mathcal C$ is a branch of nontrivial solutions to \eqref{eq:Lap} bifurcating from the critical point $(\alpha_0,0) \in \mathcal U \cap \Lambda$, then one has the following alternative:
\begin{enumerate}[label=$(\alph*)$]
\item \label{alt_a}  either $\mathcal C \cap \partial \mathcal U \neq \emptyset$;
    \item \label{alt_b} or there exists a finite set
    \begin{align*}
        \mathcal C \cap \Lambda = \{ (\alpha_0,0),(\alpha_1,0), \ldots, (\alpha_n,0) \},
    \end{align*}
    satisfying the following relation
    \begin{align*}
\sum\limits_{k=0}^n \omega_{G}(\alpha_k) = 0.
    \end{align*}
\end{enumerate} 
\end{theorem}
\vs

\begin{remark}
Suppose that Theorem \ref{th:Kras} is used to demonstrate the existence of a branch $\mathcal C$ of nontrivial solutions to \eqref{eq:Lap} bifurcating from a critical point $(\alpha_0,0)$ and that certain conditions are met such that, for any open bounded $G$-invariant neighborhood $\mathcal U \ni (\alpha_0,0)$ with $\partial \mathcal U \cap \Lambda = \emptyset$, the alternative \ref{alt_b} is impossible. Then, according to Theorem \ref{th:Rabinowitz-alt}, the branch $\mathcal C$ must be unbounded.
\end{remark}
Without an appropriate {\it fixed point reduction} of the bifurcation problem \eqref{eq:operator-eq_param}, we will be unable to guarantee that a branch of non-trivial solutions $\mathcal C$ to \eqref{eq:Lap}  bifurcating from a given isolated critical point $(\alpha_0,0) \in \br \times \mathscr H$ whose existence has been established using a Krasnosel'skii type result (e.g. by Theorem \ref{th:Kras}) does not contain {\it radial solutions}. 
With this in mind, consider the subgroup
$\bm K:=\{(-1,-1),(1,1)\}\le O(2)\times \bz_2$,
denote by $\mathscr H^{\bm K}$ the $\bm K$-fixed point space in $\mathscr H$ and by $\mathscrbf F^{\bm K}:\br\times \mathscr H^{\bm K}\to \mathscr H^{\bm K}$ the restriction 
\begin{equation}\label{eq:H-map}
\mathscrbf F^{\bm K}:=\mathscr F|_{\br\times \mathscr H^{\bm K}}.
\end{equation}
Clearly, any solution $(\alpha,u) \in \br \times \mathscr H^{\bm K}$ to the equation 
\begin{equation}\label{eq:H-Lap}
\mathscrbf F^{\bm K}(\alpha,u)=0,
\end{equation}
is also solution to \eqref{eq:operator-eq_param}. 
Notice also that shifting the angle of any $u \in \mathscr H^{\bm K}$ by half a period has the same effect as negating $u$, i.e. if $ u \in \mathscr H^{\bm K}$, then $u$ must satisfy the invariance
\[
-u(r,\theta + \pi) = u(r,\theta).
\]
Therefore, the set of radial solutions to \eqref{eq:H-Lap} is contained in the set of trivial solutions $M$. In other words, we are guaranteed that any branch of of non-trivial solutions to the bifurcation problem \eqref{eq:H-Lap} consists solely of non-radial solutions. In this $\bm K$-fixed point setting, we must adapt each of the notions used to describe the Rabinowitz alternative for the equation \eqref{eq:operator-eq_param} to the bifurcation map \eqref{eq:H-map}. To begin, notice that the $\bm K$-fixed point space $ \mathscr H^{\bm K}$ is an isometric Hilbert representation of the group $\bm G:=N(\bm K)/\bm K \simeq O(2)$ with the $\bm G$-isotypic decomposition 
\begin{align*} 
%\label{N(K)-isotypic_decomposition}
   \mathscr H^{\bm K} = \overline{\bigoplus\limits_{m=1}^\infty \mathscr H_{2m-1}}, \quad \mathscr H_{2m-1}:=\overline{\bigoplus\limits_{n=1}^\infty \mathscr E_{n,2m-1}}. 
\end{align*}
We denote by $\mathscr S^{\bm K}$ the set of $\bm K$-fixed non-trivial solutions to \eqref{eq:operator-eq_param}, i.e.
\[
\mathscr S^{\bm K} :=\{(\alpha,u)\in \br\times \mathscr H^{\bm K}:\mathscrbf F^{\bm K}(\alpha,0)=0\;\; \text{ and } \;\; u\not=0\},
\]
and by $\mathscrbf A^{\bm K} : \br \times \mathscr H^{\bm K}\to \mathscr H^{\bm K}$ the restriction of the operator $\mathscrbf A: \br \times \mathscr H \rightarrow \mathscr H$ to $\mathscr H^{\bm K}$, i.e. for each $\alpha \in \br$ we put
\begin{equation}\label{H-Linearization}
   \mathscrbf A^{\bm K}(\alpha): =\mathscrbf A(\alpha)|_{\mathscr H^{\bm K}}: \mathscr H^{\bm K}\to \mathscr H^{\bm K}. 
\end{equation}
As before, the {\it critical set} of \eqref{eq:H-Lap}, now denoted $\Lambda^{\bm K}$, is the set of trivial solutions $(\alpha_0,0) \in M$ for which $\mathscrbf A^{\bm K}(\alpha_0)$ is not an isomorphism
\[
\Lambda^{\bm K}:=\{(\alpha,0)\in \br\times \mathscr H: \mathscrbf A^{\bm K}(\alpha):\mathscr H^{\bm K}\to \mathscr H^{\bm K}\;\; \text{ is not an isomorphism}\},
\]
and the spectrum of \eqref{H-Linearization} can be described in terms of the spectra $\sigma(\mathscrbf A_{m}(\alpha))$ as follows
\[
\sigma(\mathscrbf A^ {\bm K}(\alpha))=\bigcup_{m=1}^\infty \sigma(\mathscrbf A_{2m-1}(\alpha)) = \{ \bm \mu_{m,n,j}(\alpha) \in \sigma (\mathscrbf A(\alpha)):
 2 \nmid m \}.
\]
Assume that for a given $\alpha \in \br$, the operator $\mathscrbf A^ {\bm K}(\alpha): \mathscr H^{\bm K} \rightarrow \mathscr H^{\bm K}$ is an isomorphism. If we refine the index set \eqref{def:index_set_Sigma_parameterized} to include only those indices $(n,m,j) \in \Sigma_0(\alpha)$ relevant to the $\bm K$-fixed point setting with the notation
\[
\Sigma^{\bm K}(\alpha) := \{ (n,m,j) \in \Sigma_0(\alpha) : 2 \nmid m \},
\]
then the $\bm G$-equivariant degree $\bm G\text{-deg}(\mathscrbf A^ {\bm K}(\alpha), B(\mathscr H^{\bm K}))$ can be computed as follows
\begin{align*}
    \bm G\text{-deg}(\mathscrbf A^ {\bm K}(\alpha), B(\mathscr H^{\bm K})) = \prod\limits_{(n,m,j) \in \Sigma^{\bm K}(\alpha)} \wt{\deg}_{\cW^-_{m}},
\end{align*}
where, to distinguish between $G$-basic degrees and $\bm G$-basic degrees, we have introduced the notation
\[
\wt{\deg}_{\cW^-_m}:=\bm G\text{-deg}(-\id,B(\cW^-_m)), \quad 2 \nmid m.
\]
At this point, the computational formula \eqref{eq:computational_formula_local_bifurcation_invariant} for the $\bm K$-fixed local bifurcation invariant 
\[
\omega_{\bm G}(\alpha_0) :=\bm G\text{-deg}(\mathscrbf A^{\bm K}(\alpha_0^-), B(\mathscr H^{\bm K}))-\bm G\text{-deg}(\mathscrbf A^{\bm K}(\alpha_0^+), B(\mathscr H^{\bm K})),
\]
at any isolated critical point $(\alpha_0,0)\in \Lambda^{\bm K}$ with deleted regular neighborhood $\alpha_0^- < \alpha_0 < \alpha_0^+$ becomes
\begin{align}\label{eq:loc-bif-Lap-K}
\omega_{\bm G}(\alpha_0) = \prod\limits_{(n,m,j) \in \Sigma^{\bm K}(\alpha^-)} \wt{\deg}_{\cW^-_m} - \prod\limits_{(n,m,j) \in \Sigma^{\bm K}(\alpha^+)}\wt{ \deg}_{\cW^-_m}.
\end{align}   
Likewise, Theorem \ref{th:Rabinowitz-alt} can be reformulated as follows:
\begin{theorem}\label{th:Rabinowitz-alt-K}
Let $\mathcal U \subset \br \times \mathscr H^{\bm K}$  be an open bounded $\bm G$-invariant set  with $\partial \mathcal U \cap \Lambda^{\bm K} = \emptyset$. If $\mathcal C$ is a branch of nontrivial solutions to \eqref{eq:H-Lap} bifurcating from an isolated critical critical point $(\alpha_0,0) \in \mathcal U \cap \Lambda$, then one has the following alternative:
\begin{enumerate}[label=$(\alph*)$]
\item  either $\mathcal C \cap \partial \mathcal U \neq \emptyset$;
    \item\label{alt_b} or there exists a finite set
    \begin{align*}
        \mathcal C \cap \Lambda^{\bm K} = \{ (\alpha_0,0),(\alpha_1,0), \ldots, (\alpha_n,0) \},
    \end{align*}
    satisfying the following relation
    \begin{align*}
        \sum\limits_{k=0}^n \omega_{\bm G}(\alpha_k) = 0.
    \end{align*}
\end{enumerate} 
\end{theorem}
\noindent To further simplify our exposition, we replace assumption \ref{b} with:
\begin{enumerate} [label=($\wt{B}$)] 
\item\label{b_tilde} The critical set $\Lambda \subset M$ is finite.  
\end{enumerate}
Under assumption \ref{b_tilde}, the critical set \eqref{eq:critical} can be enumerated
\begin{align}\label{eq:enumeration_critical_set}
\Lambda = \{(\alpha_0,0), (\alpha_1,0), \ldots, (\alpha_N,0) \},    
\end{align}
for some $N \in \bn$ in such a way that, if $i < j$, then $\alpha_i < \alpha_j$ for all $i,j \in \{0,1,\ldots,N\}$. Let's adopt the notations \eqref{def:set_S}--\eqref{def:set_Sigma_m} 
and account for the change in cardinality of each of the index sets $\Sigma^m(\alpha_0^-)$ and 
$\Sigma^m(\alpha_N^-)$ with the notations
\begin{align}\label{def:N^m}
\begin{cases}
    \mathfrak t^m_\Lambda := \mathfrak n^m(\alpha_N^+) - \mathfrak n^m(\alpha_0^-); \\
    J_\Lambda := \{ m \in \bn : 2 \nmid \mathfrak t^m_\Lambda \}.
\end{cases}
\end{align}
We can now exactly determine the coefficient standing next to any of the orbit types $(H_{m_0}) \in \Phi_0(G)$ in the sum of local bifurcation invariants $\{\omega_G(\alpha_k)\}_{k = 1}^N \subset A(G)$ as follows: 
\begin{proposition} \label{prop:main_global_bifurcation}
Adopting the enumeration of the critical set \eqref{eq:enumeration_critical_set} and for any $m_0 > 0$, one has
\begin{align*} 
\operatorname{coeff}^{H_{m_0}} & \left( \sum_{k=0}^N \omega_G(\alpha_k)  \right)  =  - \left[m_0 \in  J_\Lambda \right] (-1)^{[2 \nmid \mathfrak n^{m_0}(\alpha_N^+)]} \\
& \quad +2 \sum\limits_{\substack{ I \in \mathcal P(J_\Lambda) \\ I \neq \emptyset, \{m_0\} }} (-2)^{\vert I \vert-2} 
\left[ \mathcal B_H(I) \right]\left[m_0 = \operatorname{gcd}(I)  \right] (-1)^{[I \in \mathcal P(S(\alpha_N^+))]},
\end{align*}
\end{proposition}
\begin{proof}
Due to the continuous dependence of the spectrum $\sigma(\mathscrbf A(\alpha))$ on the bifurcation parameter, one has for any two adjacent critical points $(\alpha_k,0), (\alpha_{k+1},0) \in \Lambda$ (here $0 \leq k < N $) with corresponding regular neighborhoods $[\alpha_k^-,\alpha_k^+] \setminus \{\alpha_k\}$ and $[\alpha_{k+1}^-,\alpha_{k+1}^+] \setminus \{\alpha_{k+1}\}$, the coincidence $S(\alpha_k^+) = S(\alpha_{k+1}^-)$. Consequently, the computational formula suggested by Proposition \ref{prop:main_existence} telescopes into
\begin{align*}
\operatorname{coeff}^{H_{m_0}}  \left( \sum_{k=0}^N \omega_G(\alpha_N)  \right) & =  - \left[m_0 \in  S(\alpha_0^-) \right] + 2\sum\limits_{\substack{ I \in \mathcal P(S(\alpha_0^-)) \\ I \neq \emptyset, \{m_0\} }} (-2)^{\vert I \vert-2} 
\left[ \mathcal B_H(I) \right]\left[m_0 = \operatorname{gcd}(I)  \right] \\
& \quad + \left[m_0 \in  S(\alpha_N^+) \right]  -2\sum\limits_{\substack{ I \in \mathcal P(S(\alpha_N^+)) \\ I \neq \emptyset, \{m_0\} }} (-2)^{\vert I \vert-2} 
\left[ \mathcal B_H(I) \right]\left[m_0 = \operatorname{gcd}(I)  \right]. 
\end{align*}
From here, the result follows using the same arguments employed in the proof of Proposition \ref{prop:main_local_bifurcation}.
\end{proof}

\appendix
\section{The $\mathcal G$-Equivariant Degree}\label{sec:appendix}
\noi{\bf  Equivariant notation.}
Let $\mathcal G$ be a compact Lie group. For any subgroup  $H \leq \mathcal G$, we denote by $(H)$ its conjugacy class,
by $N(H)$ its normalizer by $W(H):=N(H)/H$ its Weyl group in $\mathcal G$. The set of all subgroup conjugacy classes in $\mathcal G$ 
\[
\Phi(\mathcal G):=\{(H): H\le \mathcal G\},
\]
and has a natural partial order defined as follows
\[
(H)\leq (K) \iff \exists_{ g\in \mathcal G}\;\;gHg^{-1}\leq K.
\]
As is possible with any partially ordered set, we extend the natural order over $\Phi(\mathcal G)$ to a total order, which we indicate by `$\preccurlyeq$' to differentiate the two relations. Moreover, we put
\[
\Phi_0 (\mathcal G):= \{ (H) \in \Phi(\mathcal G) \; : \; \text{$W(H)$  is finite}\},
\]
and, for any $(H),(K) \in \Phi_0(\mathcal G)$, we denote by $n(H,K)$ the number of subgroups $\tilde K \leq \mathcal G$ with $\tilde K \in (K)$ and $H \leq \tilde K$. Given a $\mathcal G$-space $X$ with an element $x \in X$, we denote by
$\mathcal G_{x} :=\{g\in \mathcal G:gx=x\}$ the {\it isotropy group} of $x$
and we call $(\mathcal G_{x}) \in \Phi(\mathcal G)$  the {\it orbit type} of $x \in X$. We also put 
\begin{align*}
    \begin{cases}
       \Phi(\mathcal G,X) := \{(H) \in \Phi(\mathcal G)  : 
(H) = (\mathcal G_x) \; \text{for some $x \in X$}\}, \\
\Phi_0(\mathcal G,X):= \Phi(\mathcal G,X) \cap \Phi_0(\mathcal G).
    \end{cases}
\end{align*}
Given any subgroup $H\leq \mathcal G$, we call the subspace 
\[
X^{H} :=\{x\in X:\mathcal G_{x}\geq H\},
\]
the {\it $H$-fixed-point subspace} of $X$. If $Y$ is another $\mathcal G$-space, then a continuous map $f : X \to Y$ is said to be {\it $\mathcal G$-equivariant} if $f(gx) = gf(x)$ for each $x \in X$ and $g \in \mathcal G$.
\vs
\noi{\bf The Burnside Ring and Axioms of the $\mathcal G$-Equivariant Brouwer Degree.}
We call the free $\mathbb{Z}$-module $A(\mathcal G) := \mathbb{Z}[\Phi_0(\mathcal G)]$ the {\it Burnside Ring} when it is equipped with the multiplicative operation
\begin{align} \label{def:burnside_product}
    (H) \cdot (K) := \sum\limits_{(L) \in \Phi_0(\mathcal G)} n_L(L), \quad (H),(K) \in \Phi_0(\mathcal G), 
\end{align}
where the coefficients $n_L \in \mathbb{Z}$ are given by the recurrence formula
\begin{align} \label{def:recurrence_formula_coefficients_burnside_product}
    n_L := \frac{n(L,H) |W(H)| n(L,K) |W(K)| - \sum_{(\tilde L) \preccurlyeq (L)} n_{\tilde L} n(L,\tilde L) |W(\tilde L)|}{|W(L)|}.
\end{align}
Since every Burnside Ring element $a \in A(\mathcal G)$ can be expressed as a formal sum over some finite number of generator elements  
\[
a = n_{H_1}(H_1) + n_{H_2}(H_2) + \cdots + n_{H_N}(H_N),
\]
we can use the notation
\[
\operatorname{coeff}^H: A(\mathcal G) \rightarrow \bz, \quad
\operatorname{coeff}^H(a) := n_H,
\]
to specify the integer coefficient standing next to the generator element $(H) \in \Phi_0(\mathcal G)$.
\vs
Let $V$ be an orthogonal $\mathcal G$-representation and suppose that $\Om \subset V$ is an open bounded $\mathcal G$-invariant set. A $\mathcal G$-equivariant map $f:V \rightarrow V$ is said to be $\Om$-admissible if $f(x) \neq 0$ for all $x \in \partial \Om$, in which case the pair $(f,\Om)$ is called an {\it admissible $\mathcal G$-pair} in $V$. We denote by $\mathcal M^{\mathcal G}(V)$ the set of all admissible $\mathcal G$-pairs in $V$ and by $\mathcal{M}^{\mathcal G}$ the set of all admissible $\mathcal G$-pairs defined by taking a union over all orthogonal $\mathcal G$-representations, i.e.
\[
\mathcal M^{\mathcal G} := \bigcup\limits_V \mathcal M^{\mathcal G}(V).
\]
The $\mathcal G$-equivariant Brouwer degree provides an algebraic count of solutions, according to their symmetric properties, to equations of the form
\[
f(x) = 0, \; x \in \Omega,
\]
where $(f, \Omega) \in \mathcal M^{\mathcal G}$. In fact, it is standard (cf. \cite{AED}, \cite{book-new}) to define the {\it $\mathcal G$-equivariant Brouwer degree} as the unique map associating to every admissible $\mathcal G$-pair $(f,\Om)\in \mathcal M^{\mathcal G}$ an element from the Burnside Ring $A(\mathcal G)$, satisfying the four {\it degree axioms} of existence, additivity, homotopy and normalization:
\vs
\begin{theorem} \rm
\label{thm:GpropDeg} There exists a unique map $\caldeg:\mathcal{M}
	^{\mathcal G}\to A(\mathcal G)$, that assigns to every admissible $\mathcal G$-pair $(f,\Omega)$ the Burnside Ring element
	\begin{equation}
		\label{eq:G-deg0}\caldeg(f,\Omega)=\sum_{(H) \in \Phi_0(\mathcal G)}%
		{n_{H}(H)},
	\end{equation}
	satisfying the following properties:
	\begin{itemize}
		\item[] \textbf{(Existence)} If  $n_{H} \neq0$ for some $(H) \in \Phi_0(\mathcal G)$ in \eqref{eq:G-deg0}, then there
		exists $x\in\Omega$ such that $f(x)=0$ and $(\mathcal G_{x})\geq(H)$.
		\item[] \textbf{(Additivity)} 
  For any two  disjoint open $\mathcal G$-invariant subsets
  $\Omega_{1}$ and $\Omega_{2}$ with
		$f^{-1}(0)\cap\Omega\subset\Omega_{1}\cup\Omega_{2}$, one has
		\begin{align*}
\caldeg(f,\Omega)=\caldeg(f,\Omega_{1})+\caldeg(f,\Omega_{2}).
		\end{align*}
		\item[] \textbf{(Homotopy)} For any 
  $\Omega$-admissible $\mathcal G$-homotopy, $h:[0,1]\times V\to V$, one has
		\begin{align*}
\caldeg(h_{t},\Omega)=\mathrm{constant}.	\end{align*}
		\item[] \textbf{(Normalization)}
  For any open bounded neighborhood of the origin in an orthogonal $\mathcal G$-representation $V$ with the identity operator $\id:V \rightarrow V$, one has
		\begin{align*}
	\caldeg(\id,\Omega)=(\mathcal G).
	\end{align*}
	\end{itemize}
 \vs
The following are additional properties of the map $\caldeg$ which can be derived from the four axiomatic properties defined above (cf. \cite{AED}, \cite{book-new}):		
\begin{itemize}
		\item[] {\textbf{(Multiplicativity)}} For any $(f_{1},\Omega
		_{1}),(f_{2},\Omega_{2})\in\mathcal{M} ^{\mathcal G}$,
		\begin{align*}
			\caldeg(f_{1}\times f_{2},\Omega_{1}\times\Omega_{2})=
		\caldeg(f_{1},\Omega_{1})\cdot \caldeg(f_{2},\Omega_{2}),
		\end{align*}
		where the multiplication `$\cdot$' is taken in the Burnside ring $A(\mathcal G )$.

		\item[] \textbf{(Recurrence Formula)} For an admissible $\mathcal G$-pair
		$(f,\Omega)$, the $\mathcal G$-degree \eqref{eq:G-deg0} can be computed using the
		following Recurrence Formula:
		\begin{equation}
			\label{eq:RF-0}n_{H}=\frac{\deg(f^{H},\Omega^{H})- \sum_{(K)\preccurlyeq(H)}{n_{K}\,
					n(H,K)\, \left|  W(K)\right|  }}{\left|  W(H)\right|  },
		\end{equation}
		where $\left|  X\right|  $ stands for the number of elements in the set $X$
		and $\deg(f^{H},\Omega^{H})$ is the Brouwer degree of the map $f^{H}%
		:=f|_{V^{H}}$ on the set $\Omega^{H}\subset V^{H}$.
	\end{itemize}
\end{theorem}
The natural generalization of the $\mathcal G$-equivariant Brouwer degree to its infinite dimensional counterpart the $\mathcal G$-equivariant Leray-Schauder degree is described in detail elsewhere (see for example \cite{survey, book-new, AED}).
\vs
\noi{\bf Computational Formulae for the $\mathcal G$-Equivariant Brouwer Degree.} 
We denote by $\{ \mathcal V_i \}_{i \in \mathbb{N}}$ the set of all irreducible $\mathcal G$-representations and define the $i$-th basic degree as follows
\begin{align*}
\deg_{\mathcal{V}_{i}}:=\caldeg(-\id,B(\mathcal{V} _{i})).
\end{align*} 
Given any orthogonal $\mathcal G$-representation with a $\mathcal G$-isotypic decomposition
\[
V = \bigoplus_{i \in \mathbb{N}} V_i,
\]
and any $\mathcal G$-equivariant  linear isomorphism $T:V\to V$, the Multiplicativity and Homotopy properties of the $\mathcal G$-equivariant Brouwer degree, together with Schur's Lemma implies
\begin{align*}
%\label{eq:prod-prop}
  \caldeg(T,B(V))=\prod_{i \in \mathbb{N}} \caldeg
	(T_{i},B(V_{i}))= \prod_{i \in \mathbb{N}}\prod_{\mu\in\sigma_{-}(T)} \left(
	\deg_{\mathcal{V} _{i}}\right)  ^{m_{i}(\mu)}%  
\end{align*}
where $T_{i}=T|_{V_{i}}$, $\sigma_{-}(T)$ denotes the real negative
spectrum of $T$ and $m_i(\mu) := \dim E_i(\mu)/ \dim \mathcal V_i$ (here, we indicate by
$E(\mu)$ the generalized eigenspace associated with any $\mu \in \sigma(T)$ and $E_i(\mu) := E(\mu) \cap V_i$).
\vskip.3cm
Notice that each of the basic degrees: 
\begin{align*}
	\deg_{\mathcal{V} _{i}}=\sum_{(H)}n_{H}(H),
\end{align*}
can be practically computed, using the recurrence formula  \eqref{eq:RF-0}, as follows
\begin{align*}
n_{H}=\frac{(-1)^{\dim\mathcal{V} _{i}^{H}}- \sum_{(H)\preccurlyeq(K)}{n_{K}\, n(H,K)\, \left|  W(K)\right|  }}{\left|  W(H)\right|  }.
\end{align*}

\end{document}